\documentclass[ECP,preprint]{ejpecp} 

\usepackage{units} 
\let\origemptyset\emptyset
\usepackage{mathabx} 
\let\emptyset\origemptyset

\SHORTTITLE{Lonely passengers on buses}

\TITLE{Lonely passenger problem: the more buses there are, the more lonely passengers there will be}

\AUTHORS{%
  Imre Péter Tóth\footnote{Department of Stochastics,
Budapest University of Technology and Economics;
Egry József utca 1, H-507; H-1111 Budapest, Hungary
    \EMAIL{mogy@math.bme.hu}}
  }

\KEYWORDS{Markov chain ; stochastic dominance ; coupling ; reverse Markov chain ; bus problem ; Stirling numbers of the second kind} 

\AMSSUBJ{60J10 ; 60C05} 

\SUBMITTED{January 10, 2025} 
\ACCEPTED{???} 

\VOLUME{0}
\YEAR{2023}
\PAPERNUM{0}
\DOI{10.1214/YY-TN}

\ABSTRACT{Empty buses are standing at a bus station. $n$ passengers arrive, and they each board a bus completely at random (meaning that they choose uniformly and independently). Then all buses depart. We show that the more buses there are, the more likely it is that someone (i.e. at least one passenger) travels alone (while $n$ is fixed). More generally, we show that the number of lonely passengers increases with the number of buses, in the sense of stochastic dominance. This problem turned out to be surprisingly difficult, with no short solution known to the author so far, despite the efforts of many experts. Some of the results can also be formulated as properties of Stirling numbers of the second kind.}

\newtheorem{convention}[theorem]{Convention}

\newcommand{\IR}{\mathbb{R}}
\newcommand{\IN}{\mathbb{N}}
\newcommand{\prob}{\mathbb{P}}
\newcommand{\expect}{\mathbb{E}}
\newcommand{\sigalg}{\mathcal{F}}
\newcommand{\NE}{\mathrm{NE}}
\newcommand{\rev}{\mathrm{rev}}
\newcommand{\sgt}{\succ} 
\newcommand{\sle}{\preccurlyeq} 
\newcommand{\sge}{\succcurlyeq} 

\makeatletter
\def\moverlay{\mathpalette\mov@rlay}
\def\mov@rlay#1#2{\leavevmode\vtop{%
   \baselineskip\z@skip \lineskiplimit-\maxdimen
   \ialign{\hfil$\m@th#1##$\hfil\cr#2\crcr}}}
\newcommand{\charfusion}[3][\mathord]{
    #1{\ifx#1\mathop\vphantom{#2}\fi
        \mathpalette\mov@rlay{#2\cr#3}
      }
    \ifx#1\mathop\expandafter\displaylimits\fi}
\makeatother

\newcommand{\bigcupdot}{\charfusion[\mathop]{\bigcup}{\cdot}}

\begin{document}




\section{Introduction}

Let us seat $n\ge 1$ passengers independently and uniformly on $k\ge 1$ buses, and let $p_{n,k}$ denote the probability that at least one passenger travels alone. We prove that $p_{n,k+1}>p_{n,k}$ for each $n>1$ and $k$. More generally, if $L^{(k)}_n$ is the number of lonely passengers in the case of $k$ buses and $n$ passengers, then $L^{(k+1)}_n$ stochastically dominates $L^{(k)}_n$.

This problem could also be naturally formulated in term of balls and bins: We place $n$ balls into $k$ bins uniformly and independently. Show that the probability that there is at least one bin containing exactly one ball is increasing in $k$.

However, I will stick to the language of passengers on buses, mainly for historical reasons. This is how the problem was first formulated, and this is how it was discussed by people working on it for several months. I also find that for this problem, this language is at least as intuitive as that of balls and bins.

The motivation, apart from the naturality of the question, is a problem from László Márton Tóth in probabilistic graph theory. That problem, still unsolved, is presented briefly in Section~\ref{sec:motivation}. The lonely passenger problem was designed to feature a key difficulty of that problem in the cleanest possible form.

Some of the results of this paper can be formulated in terms of Stirling numbers of the second kind, as pointed out to me by Ed Crane, and later by Péter Csikvári. This is discussed briefly in Section~\ref{sec:Stirling}.

The problem, first formulated in September 2023, turned out to be surprisingly difficult, especially in contrast with the simplicity and intuitiveness of the statement. It was discussed by quite a few experts of probability for about half a year, with many wrong solutions born. Conditional probabilities are tricky, and intuition is often misleading. This is one of the reasons that the proof here is presented with many seemingly obvious details written out, to be closer to the safe side.

The basic idea of the proof is to look at the systems with $k$ and $k+1$ buses as passengers arrive one by one, and couple the two processes. This idea comes (to me) from Balázs Ráth. 

\subsection{Statement of the result}
\label{sec:statement}

Let the $n$ passengers of the $k$ buses arrive one by one. Formally, for $m=1,2,\dots,n$ let $B^{(k)}_m$ be independent and uniformly distributed on $\{1,2,\dots,k\}$ with $B^{(k)}_m$ denoting the number of the bus (out of $k$) taken by passenger number $m$. 

Let $L^{(k)}_n$ denote the number of lonely passengers, which is clearly equal to the number of buses with exactly one passenger:
\begin{align}
 L^{(k)}_n & :=\#\left\{u\in\{1,2,\dots,n\}\,\middle|\, B^{(k)}_u \neq B^{(k)}_v \text{ if }u\neq v\in\{1,2,\dots,n\}\right\} \\
 & = \#\left\{l\in\{1,2,\dots,k\}\,\middle|\, \exists! u\in\{1,2,\dots,n\right\} \text{ such that } B^{(k)}_u=l\ \}.
\end{align}
So $p_{n,k}:=\prob(L^{(k)}_n>0)$ is the probability that at least one passenger travels alone.

We use $\sge$ and $\sgt$ to denote stochastic dominance between (real valued) random variables: $X\sge Y$ if $\prob(X\ge u) \ge \prob(Y\ge u)$ for every $u\in\IR$. $X\sgt Y$ if $X\sge Y$ but $X$ and $Y$ are not identically distributed.

\begin{theorem}\label{thm:L_k_n_dominance}
 For every $n>1$ and $k\ge 1$, $L^{(k+1)}_n \sgt L^{(k)}_n$. Also $p_{n,k+1}>p_{n,k}$.
\end{theorem}

(Clearly for $n=1$ we have $L^{(k)}_1=1$ and $p_{1,k}=1$, so $L^{(k+1)}_n \sge L^{(k)}_n$ and $p_{n,k+1}\ge p_{n,k}$ still hold.)

\subsection{Structure of the proof}

We first discuss versions of the model when all buses are required to be nonempty, or exactly $l$ buses are required to be nonempty, and the relation of these models to each other.

The basic step, repeated many times, is to look at the time evolution of different quantities in different models as the passengers arrive, and couple. This idea comes (to me) from Balázs Ráth.

The first thing we show is that the number of nonempty buses increases (stochastically) in $k$. This reduces the problem to the analogous problem with the condition that no bus (out of $l$) can be empty. This reduction is common knowledge among people who have been fighting with the bus problem, but I write out details to be sure.

The reduced problem, with no empty buses, is then discussed in several steps: 
\begin{enumerate}
 \item First we show that the number of nonempty buses at intermediate times $0<m<n$ also increases (stochastically) in $l$. This is done by looking at the time reversed versions of the Markov chains describing the evolution of nonempty buses in the different models. 
 \item Second, we use this to show that the probability of the first passenger travelling alone is increasing in $l$. I'm somewhat surprised that I couldn't find an easier / elementary proof of this statement, which is trivial in the original model. However, since then, Péter Csikvári gave an elegant (although still non-trivial) combinatorial proof -- see Section~\ref{sec:Stirling}.
 \item Third, we construct a very strong coupling of the nonempty bus count processes (as passengers arrive) in the $l$ and $l+1$ bus models, ensuring that the two processes do exactly the same apart from a single step when the big one grows and the small one doesn't. This is again done using the reverse Markov chains.
 \item Now the lonely passenger count processes can eventually be coupled.
\end{enumerate}

\section{Further models and notation}\label{sec:model}

\subsection{Given number of buses}

Sometimes it is convenient to look at a specific realization: Let $\Omega^{(k,n)}=\{1,2,\dots,k\}^n$, let $\prob_{k,n}$ be the uniform probability on $\Omega^{(k,n)}$. Let $B^{(k)}_1,\dots,B^{(k)}_n:\Omega^{(k,n)}\to\IN$ be given by $B^{(k)}_m(\underline{\omega}):=\omega_m$, or just say that $B^{(k,n)}_\bullet=(B^{(k)}_1,\dots,B^{(k)}_n):\Omega^{(k,n)}\to\IN^n$ is the identity function.

\begin{notation}
 For a vector $(a_1,a_2,\dots,a_n)\in\IR^n$ we use the notation $\underline{a}$ or $a_\bullet$ depending on the context -- whether we want to think of the sequence as a function of time or just as a vector.
 
 For vectors in $\IR^n$ we use $\ge$ and $\le$ to denote the usual partial ordering given by elementwise comparison.
\end{notation}

\begin{notation}[Stochastic dominance]
As mentioned in Section~\ref{sec:statement}, we use $\sge$ to denote (non-strict, first order) stochastic dominance between ($\IN$-valued) random variables: $X\sge Y$ if $\prob(X\ge i)\ge \prob(Y\ge i)$ for every $i\in\IN$. This is equivalent to the existence of a coupling $(\tilde{X},\tilde{Y})$ such that $\tilde{X}\sim X$, $\tilde{Y}\sim Y$ and $(\tilde{X}\ge\tilde{Y})$ almost surely.

With some abuse of notation we will also use $\sge$ to compare stochastic processes on $\IN$: we write $X_\bullet\sge Y_\bullet$ if there is a coupling $(\tilde{X}_\bullet,\tilde{Y}_\bullet)$ of the processes such that almost surely $\tilde{X}_m\ge \tilde{Y}_m$ for every $m$.
\end{notation}

For $m=0,1,\dots,n$ let $L^{(k)}_m$ denote the number of lonely passengers after the arrival of the $m$th passenger:
\begin{align}
 L^{(k)}_m & :=\#\left\{u\in\{1,2,\dots,m\}\,\middle|\, B^{(k)}_u \neq B^{(k)}_v \text{ if }u\neq v\in\{1,2,\dots,m\}\right\} \\
 & = \#\left\{l\in\{1,2,\dots,k\}\,\middle|\, \exists! u\in\{1,2,\dots,m\right\} \text{ such that } B^{(k)}_u=l\ \}.
\end{align}
(So $L^{(k)}_0=0$.)

For $m=0,1,\dots,n$ let $N^{(k)}_m$ denote he number of nonempty buses after the arrival of the $m$th passenger:
\begin{equation}\label{eq:N(k)_m_def}
 N^{(k)}_m=\#\left\{B^{(k)}_1,\dots,B^{(k)}_m\right\}.
\end{equation}
(So $N^{(k)}_0=\# \emptyset = 0$.)

\subsection{Strict versus non-strict inequalities}

The following lemma helps reduce the proofs of strict inequalities between probabilities to non-strict ones.

\begin{lemma}\label{lem:prob_not_equal}
Let $A\subset \Omega^{(k,n)}$ and $B\subset \Omega^{(k+1,n)}$ be any events. Then $\prob_{k,n}(A)\neq \prob_{k+1,n}(B)$ unless both of them are $0$ or $1$.
\end{lemma}

\begin{proof}
 $p_A:=\prob_{k,n}(A)=\frac{|A|}{k^n}$ is a multiple of $\frac{1}{k^n}$, while 
 $p_B:=\prob_{k+1,n}(B)=\frac{|B|}{(k+1)^n}$ is a multiple of $\frac{1}{(k+1)^n}$. Such probabilities can hardly be equal:
 If $p_A=p_B$, then $|A|(k+1)^n=|B| k^n$, so $|A| (k+1)^n$ is a multiple of $k^n$. But $k^n$ and $(k+1)^n$ are relatively prime, so $|A|$ must also be a multiple of $k^n$, meaning that $p_A=p_B$ is an integer.
\end{proof}

Keeping this lemma in mind, we only aim at the non-strict version of Theorem~\ref{thm:L_k_n_dominance}, and work with non-strict inequalities throughout the paper, which saves us from worries about special cases. Then the strict version follows automatically, since the events in question are non-trivial.

\subsection{No empty buses}
\label{sec:noempty}

Assume we have $k$ buses and $n\ge k$ passengers. Then we can restrict to the configurations when there is no empty bus. Let $\Omega^{(k,n,\NE)}\subset \Omega^{(k,n)}$ be the set of such configurations (NE stands for ``no empty''), which can be written in several convenient forms:
\begin{align}
 \Omega^{(k,n,\NE)} & := \{N^{(k)}_n=k\}=\{\{B^{(k)}_1,\dots,B^{(k)}_n\}=\{1,2,\dots,k\}\} \\ 
 & = \{\underline{\omega}\in\Omega^{(k,n)}\,|\,\{\omega_1,\dots,\omega_n\}=\{1,2,\dots,k\}\}.
\end{align}
Let $\prob^\NE_{k,n}$ be the uniform probability measure on $\Omega^{k,n,\NE}$. Since $\Omega^{(k,n,\NE)}\subset \Omega^{(k,n)}$, all events and random variables defined on the probability space $(\Omega^{(k,n)},\prob_{k,n})$ make sense on $\left(\Omega^{(k,n,\NE)},\prob^\NE_{k,n}\right)$ as well, but the probabilities and distributions are of course different:
\begin{equation}\label{eq:prob-NE-A}
 \prob^\NE_{k,n} (A) = \prob_{k,n}\left(A\,\middle| \,\Omega^{(k,n,\NE)}\right)
\end{equation}
for every $A\subset \Omega^{(k,n)}$ and
\begin{equation}\label{eq:prob-NE-Xx}
 \prob^\NE_{k,n} (X=x) = \prob_{k,n}\left(X=x\,\middle| \,\Omega^{(k,n,\NE)}\right)
\end{equation}
for any function $X$ on $\Omega^{(k,n,\NE)}$ (into any set $S$) and any value $x\in S$, since $\prob^\NE_{k,n}$ is exactly $\prob_{k,n}$ conditioned on $\Omega^{(k,n,\NE)}$.
\begin{remark}
 \begin{enumerate}
  \item These spaces are finite, so we always define the probabilities on the discrete $\sigma$-algebras, which we omit in the notation. Also, all random variables into any set are discrete, so the probabilities of particular values give full information on the distribution.
  \item There is a slight abuse of notation in (\ref{eq:prob-NE-A}): $\prob^\NE_{k,n} (A)$ should formally be $\prob^\NE_{k,n} (A\cap \Omega^{(k,n,\NE)})$, but this hopefully causes no confusion. On the other hand, (\ref{eq:prob-NE-Xx}) is formally correct with $\{X=x\}$ denoting different events (on different spaces) on the two sides of the equality.
  \item The probability $\prob^\NE_{k,n}$ is uniform, but its domain $\Omega^{(k,n,\NE)}$ is more complicated than $\Omega^{(k,n)}$ (e.g. there is no closed formula for the number of elements), so probabilities are often more tricky. In particular, $(B^{(k)}_1,\dots,B^{(k)}_n)$ are clearly no longer independent under $\prob^\NE_{k,n}$.
 \end{enumerate}
\end{remark}

\subsection{Fixed number of nonempty buses}

If there are $k$ buses and $n$ passengers, $1\le l\le n, k$ and we condition on the event that there are exactly $l$ nonempty buses (that is, on $\{N^{(k)}_n=l\}$), that is essentially the same as if there are $l$ buses and we condition on all of them being nonempty (that is, using the space $(\Omega^{(l,n,\NE)},\prob^\NE_{l,n})$). In some sense, if there are exactly $l$ nonempty buses, then we can assume, without loss of generality, that these nonempty buses are the first $l$ -- as long as the questions we ask are not sensitive to the numbering of the buses. This is made rigorous in the following lemma.

\begin{lemma}\label{lem:conditioning_on_N_equals_l}
 Let $1\le l \le n, k$ and let $S$ be any nonempty set.
 Assume that the function $f: \{1,2,\dots,k\}^n\to S$ is invariant under permutations:
 \begin{equation}
  f(a_1,a_2,\dots,a_n)=f(\pi(a_1),\pi(a_2),\dots,\pi(a_n))
 \end{equation}
 for any permutation $\pi:\{1,2,\dots,k\} \righttoleftarrow$.
 Then for any $s\in S$
 \begin{equation}
  \prob_{k,n} \left( f\left(B^{(k)}_1,\dots,B^{(k)}_n\right)=s \,\middle|\, N^{(k)}_n=l \right) = \prob^\NE_{l,n}\left(f\left(B^{(k)}_1,\dots,B^{(k)}_n\right)=s\right).
 \end{equation}
\end{lemma}

\begin{proof}
  If we know that exactly $l$ out of $k$ buses are nonempty, it doesn't matter which $l$ those are. Let
  \begin{equation}
   O_{k,n}=\left\{B^{(k)}_1,\dots,B^{(k)}_n\right\}
  \end{equation}
  be the (random) set of nonempty buses.
  So if $U\subset\{1,2,\dots,k\}$ and $|U|=l$, then take any permutation $\pi$ on  $\{1,2,\dots,k\}$ such that $\pi(U)=\{1,2\dots,l\}$, and use the permutation symmetry of $f$ to get
 \begin{align}
 & \prob_{k,n}\left(f\left(B^{(k)}_1,\dots,B^{(k)}_n\right)=s \,\middle|\, O_{k,n}=U\right) \\
 & =  
 \prob_{k,n}\left(f\left(\pi(B^{(k)}_1),\dots,\pi(B^{(k)}_n)\right)=s \,\middle|\, \pi(O_{k,n})=\pi(U)\right) \\
 & =  
 \prob_{k,n}\left(f\left(\pi(B^{(k)}_1),\dots,\pi(B^{(k)}_n)\right)=s \,\middle|\, \left\{\pi(B^{(k)}_1),\dots,\pi(B^{(k)}_n)\right\}=\{1,2,\dots,l\}\right).
 \end{align}
 But $\left(B^{(k)}_1,\dots,B^{(k)}_n\right)$ and $\left(\pi(B^{(k)}_1),\dots,\pi(B^{(k)}_n)\right)$ have the same joint distribution (under $\prob_{n,k}$), so 
  \begin{align}
 & \prob_{k,n}\left(f\left(\pi(B^{(k)}_1),\dots,\pi(B^{(k)}_n)\right)=s \,\middle|\, \left\{\pi(B^{(k)}_1),\dots,\pi(B^{(k)}_n)\right\}=\{1,2,\dots,l\}\right) \\
 & = 
 \prob_{k,n}\left(f\left(B^{(k)}_1,\dots,B^{(k)}_n\right)=s \,\middle|\, 
 \left\{B^{(k)}_1,\dots,B^{(k)}_n\right\} =\{1,2,\dots,l\}\right),
  \end{align}
 giving
 \begin{equation}\label{eq:prob_f_s_symmetric}
 \prob_{k,n}\left(f\left(B^{(k)}_\bullet\right)=s \,\middle|\, O_{k,n}=U\right) =  
 \prob_{k,n}\left(f\left(B^{(k)}_\bullet\right)=s \,\middle|\, O_{k,n}=\{1,2,\dots,l\}\right).
 \end{equation}
 In turn, conditioning on the first $l$ buses being nonempty is exactly the same as having $l$ buses with none of them empty:
 \begin{equation}
  \Omega^{(l,n,\NE)} = \left\{ O_{k,n}=\{1,2,\dots,l\} \right\} \subset \Omega^{(k,n)},
 \end{equation}
 $\prob_{k,n}$ conditioned on this set is exactly the uniform probability, so
 \begin{equation}\label{eq:prob_f_s_first_l}
  \prob_{k,n}\left(f\left(B^{(k)}_\bullet\right)=s \,\middle|\, O_{k,n}=\{1,2,\dots,l\}\right) = \prob^\NE_{l,n}\left(f\left(B^{(k)}_\bullet\right)=s\right).
 \end{equation}
 Putting (\ref{eq:prob_f_s_symmetric}) and (\ref{eq:prob_f_s_first_l}) together give
 \begin{equation}\label{eq:prob_f_s_O_U}
  \prob_{k,n}\left(f\left(B^{(k)}_\bullet\right)=s \,\middle|\, O_{k,n}=U\right) =
  \prob^\NE_{l,n}\left(f\left(B^{(k)}_\bullet\right)=s\right)
 \end{equation}
 for every $U\in H^{k,l}$, where
  \begin{equation}
  H^{k,l}:=\{U\subset\{1,2,\dots,k\}\,:\, |U|=l\},
 \end{equation}
 so
 \begin{equation}\label{eq:N_is_l_as_union}
  \left\{ N^{(k)}_n=l \right\} = \bigcupdot_{U\in H^{k,l}} \left\{ O_{k,n}=U \right\}.
 \end{equation}
 The statement of the lemma follows by additivity of the probability:
 \begin{align}
  \prob_{k,n} \left( f\left(B^{(k)}_\bullet\right)=s \,\middle|\, N^{(k)}_n=l \right)
  & = \frac{\prob_{k,n} \left( N^{(k)}_n=l, f\left(B^{(k)}_\bullet\right)=s \right)}{\prob_{k,n} \left(  N^{(k)}_n=l \right)} \\
  & = \frac{\sum_{U\in H^{k,l}} \prob_{k,n} (O_{k,n}=U) \prob_{k,n}\left(f\left(B^{(k)}_\bullet\right)=s \,\middle|\, O_{k,n}=U\right)}{\sum_{U\in H^{k,l}} \prob_{k,n} (O_{k,n}=U)} \\
  &= \prob^\NE_{l,n}\left(f\left(B^{(k)}_\bullet\right)=s\right)
 \end{align}
 by (\ref{eq:prob_f_s_O_U}), so the lemma is proven.
\end{proof}

\begin{remark}
 This lemma is very intuitive and quite trivial, but not as trivial as it may seem. The idea was that if we know that there are exactly $l$ nonempty buses, then we can assume, without loss of generality, that exactly buses $\{1,2,\dots,l\}$ are nonempty. But conditional probabilities are tricky. For example, one could naively think that if we know that there are \emph{at most} $l$ nonempty buses, then we can assume, without loss of generality, that \emph{at most} buses $\{1,2,\dots,l\}$ are nonempty, so e.g.
 \begin{align}
  & \prob_{k,n}(\text{first two passengers travel together}\,|\,\text{at most $l$ buses are nonempty}) \\
  & = \prob_{k,n}(\text{first two passengers travel together}\,|\,\text{at most buses $\{1,\dots,l\}$ are nonempty}).
 \end{align}
 This is not true: one can easily check for $l=2$, $n=k=3$ that the LHS is $\nicefrac{3}{7}$ and the RHS is $\nicefrac{1}{2}$.
 The proof of Lemma~\ref{lem:conditioning_on_N_equals_l} breaks down because the union in (\ref{eq:N_is_l_as_union}) is not disjoint.
\end{remark}

\begin{remark}
 The conditions on permutation symmetry in this lemma could be relaxed to cover, for example, the case when $f$ is the indicator of $\{B_1<B_2\}$.
\end{remark}

\begin{remark}
The lemma will be applied to functions $f:\{1,2,\dots,k\}^n\to S$ where $S$ is possibly huge: we will map the passenger configuration to entire stochastic processes. 
\end{remark}

\section{Preliminaries}

\subsection{Ordering of pure birth processes}

A Markov chain $X_0,\dots,X_n$ on $\IN$ is called a pure birth process if jumps can only be $0$ or $1$, so $0\le X_{m+1}-X_m\le 1$ almost surely for $m=0,1,\dots,(n-1)$. Similarly, $X_0,\dots,X_n$ is a pure death process if jumps can only be $0$ or $-1$. Note that $X_m$ can be time inhomogeneous. 

The following lemma says that if a pure birth process tends to grow faster than another, and it's initially bigger (stochastically), then it stays bigger all the time (stochastically).

\begin{lemma}~\label{lem:birth_process_dominance}
 Let $X_0,X_1,\dots,X_n$ and $Y_0,Y_1,\dots,X_n$ be possibly time inhomogeneous pure birth processes on $\IN$. Denote the birth probabilities by $r^{X,m}_i:=\prob(X_{m+1}=i+1\,|\, X_m=i)$ and $r^{Y,m}_i:=\prob(Y_{m+1}=i+1\,|\, Y_m=i)$. Assume that $r^{X,m}_i\ge r^{Y,m}_i$ for every $m=0,1,\dots,(n-1)$ and $i\in\IN$ (meaning that $X_m$ is more likely to grow than $Y_m$ whenever they are equal). Assume also that $X_0\sge Y_0$. Then $X_\bullet\sge Y_\bullet$.
\end{lemma}

\begin{proof}
 The processes can be coupled by constructing a Markov chain $(X_m,Y_m)\in\IN^2$ such that $X_0\ge Y_0$ and, whenever $X_m=Y_m$ and $Y_m$ grows, then so does $X_m$.
%
%
 Then $X_m\ge Y_m$ for all $m$ by induction. 
\end{proof}

Obvious couplings of pure birth and pure death processes similar to this one will be used several times in the paper, without writing them out as lemmas with separate proofs.

\subsection{Stochastic monotonicity of the number of nonempty buses}

Recall from (\ref{eq:N(k)_m_def}) that $N^{(k)}_m$ is the number of nonempty buses after the arrival of $m$ passengers. The following lemma says that $N^{(k)}_m$ is monotone increasing in $k$ in the sense of stochastic dominance.

\begin{lemma}\label{lem:nonempty_bus_dominance_1}
 $N^{(k+1)}_\bullet \sge N^{(k)}_\bullet$ for every $k=1,2,\dots$. In particular, $N^{(k+1)}_m \sge N^{(k)}_m$ for every $m$ and $k$.
\end{lemma}

\begin{proof}
 $N^{(k)}$ is a time homogeneous pure birth process with birth probabilities
 \begin{equation}
 r^{(k)}_i:=\prob\left(N^{(k)}_{m+1}=i+1\,\middle|\, N^{(k)}_m=i\right)=1-\frac{i}{k}
 \end{equation}
for $i=0,1,\dots,k$. (The states $i>k$ are never reached, so we can set anything, for example $r^{(k)}_i:=0$ for these.)  So $r^{(k+1)}_i\ge r^{(k)}_i$ for all $i\in\IN$, and Lemma~\ref{lem:birth_process_dominance} applies, giving exactly the statement we are proving.
\end{proof}
To my knowledge, this lemma was first proven by Márton Balázs.

\begin{remark}
 The argument of this proof can be fine tuned to construct a coupling in which $N^{(k)}_m \sle N^{(k+1)}_m \sle N^{(k)}_m+1$.
 \footnote{Equivalently: In this coupling the number of nonempty buses and the number of empty buses is simultaneously bigger in the $(k+1)$ bus system than in the $k$ bus system. (More precisely, one of them is equal and the other is bigger by $1$.)}
 Such a coupling gives fairly good comparison of the evolutions of the number of lonely passengers in the two systems, but it's just not good enough to get Theorem~\ref{thm:L_k_n_dominance} -- at least I could not do it. Instead, we will couple the evolutions of the number of nonempty buses in systems where no bus can remain empty in the end. In these models an even better coupling is possible.
\end{remark}

\subsection{Conditional Markov chains}
\label{sec:conditional_MC}

What is written in this and the next section is well known from the classical theory of Markov chains. We repeat it here to allow easy referencing, and to emphasize important details.

Let $X_0,X_1,\dots,X_n$ be a Markov chain on the finite or countable state space $S$ and let $A$ be an event depending on $X_n$ only. Then the process $X_\bullet$ conditioned on $A$ is also a Markov chain. Formally:

\begin{lemma}\label{lem:MC_conditioning}
Let $(\Omega,\sigalg,\prob)$ be a probability space, let $X_0,X_1,\dots,X_n:\Omega\to S$ be a Markov chain. Let $\Omega\supset A\in\sigma(X_n)$ and define $\prob_A:\sigalg\to [0,1]$ as $\prob_A(B):=\prob(B\,|\, A)$ for every $B\in\sigalg$. Then $X_0,X_1,\dots,X_n$ is also a Markov chain w.r.t. $\prob_A$, meaning
\begin{equation}
 \prob_A(X_{m+1}=j \, | \, X_0=i_0, X_1=i_1,\dots, X_m=i_m) = \prob_A(X_{m+1}=j \, | X_m=i_m)
\end{equation}
for every $m=0,1,\dots,(n-1)$ and $i_0,i_1,\dots,i_m,j\in S$.
\end{lemma}


The proof is an easy calculation using only the definition of Markov chains and conditional probability. Of course, the conditional process will not be time homogeneous, even if the original process was (unless the condition is trivial). The conditional process is a special case of the Doob $h$-transform, see e.g. \cite{Swart} Proposition 1.6 and the remark following it.

\subsection{Time reversed Markov chains}
\label{sec:time_reverse}

If $X_0,X_1,\dots,X_{n-1},X_n\in S$ is a Markov chain, then so is $X_n,X_{n-1},\dots,X_1,X_0$. This is obvious from the characterization of Markov chains which requires that conditioned on the present state the past and the future are conditionally independent.

It's important to keep in mind that the reversed process is typically time inhomogeneous, even if the original was time homogeneous.
Moreover, the reversed process is not specified by the transition probabilities of the forward process. On the contrary, it depends heavily on $X_0$, in the sense that not only the initial distribution (the distribution of $X_n$) depends on the distribution of $X_0$, but also the transition probabilities $\prob(X_{m-1}=j\,|\,X_m=i)$ depend on the distribution of $X_0$. Indeed, if for example $X_0=i$ is deterministic, then the transition rules of the revered process should be such that it arrives to $i$ with probability $1$, no matter where it starts from (as long as we don't try to start the reversed process from an $X_n=j$ which is not reachable from $i$ in $n$ steps).

On the other hand, it makes perfect sense to ``start the time reversed process'' from any state $x\in S$ which is possible for $X_n$. This can be pictured in two different ways:
\begin{itemize}
 \item Calculate the transition probabilities $P^{\rev,m}_{i,j}:=\prob(X_{m-1}=j\,|\, X_m=i)$ (which are clearly determined by the distribution of $X_0$, the transition probabilities of the forward chain, and nothing else). Then choose $\tilde{X}_n=x$, and build the Markov chain $\tilde{X}_n,\tilde{X}_{n-1},\dots,\tilde{X}_0$ using the transition matrices $P^{\rev,m}$.
 \item Condition the Markov chain $X_n,X_{n-1},\dots,X_0$ -- as a random sequence -- on $\{X_n=x\}$ to obtain $\tilde{X}_n,\tilde{X}_{n-1},\dots,\tilde{X}_0$. Conditioning on $\{X_n=x\}$ and time reversal commute, so this process $\tilde{X}_n,\tilde{X}_{n-1},\dots,\tilde{X}_0$ is the same as if we first condition on $A:=\{X_n=x\}$ as in Section~\ref{sec:conditional_MC}, and then take time reversal.
\end{itemize}

\section{Proof of Theorem~\ref{thm:L_k_n_dominance}}

Let $n\ge l>1$ passengers arrive one by one. The proof revolves around comparing the cases when there are exactly $l$ nonempty buses in the end, and when there are exactly $l-1$. We first show that in the $l$ case, the number of nonempty buses is higher at all times -- at least stochastically. This is formulated in the following lemma.

\begin{convention}\label{conv:Ntilde}
 Consider the number of nonempty buses $N^{(l)}_m$ after the arrival of $m$ passengers, on the probability space $(\Omega^{(l,n,\NE)},\prob^\NE_{l,n})$ as described in Section~\ref{sec:noempty}. Call this conditioned process $\tilde{N}^{(l,n)}_m$ to avoid confusion with Lemma~\ref{lem:nonempty_bus_dominance_1}.
\end{convention}
 
\begin{lemma}\label{lem:nonempty_bus_dominance_2}
  $\tilde{N}^{(l,n)}_\bullet \sge \tilde{N}^{(l-1,n)}_\bullet$
\end{lemma}

\begin{proof}
 If we apply Lemma~\ref{lem:conditioning_on_N_equals_l} to the function $f$ on $\{1,2,\dots,k\}^n$ which maps the passenger configuration to the entire nonempty bus count process $N^{(k)}_\bullet$, then we get that the processes $\tilde{N}^{(l,n)}_\bullet$ and $\tilde{N}^{(l-1,n)}_\bullet$ can be constructed by conditioning the same Markov chain $N^{(k)}_\bullet$ (for any $k\ge l$) on the events $\{N^{(k)}_n=l\}$ and $\{N^{(k)}_n=l-1\}$, respectively. This means that they are themselves Markov chains (by Lemma~\ref{lem:MC_conditioning}). Let $\overleftarrow{\tilde{N}}^{(l,n)}_m:=\tilde{N}^{(l,n)}_{n-m}$ and $\overleftarrow{\tilde{N}}^{(l-1,n)}_m:=\tilde{N}^{(l-1,n)}_{n-m}$ be their time reversed versions. Then $\overleftarrow{\tilde{N}}^{(l,n)}_\bullet$ and $\overleftarrow{\tilde{N}}^{(l-1,n)}_\bullet$ are pure death processes. Crucially, they have the same transition probabilities (as discussed in Section~\ref{sec:time_reverse}), and they are started from $\overleftarrow{\tilde{N}}^{(l,n)}_0=l>l-1=\overleftarrow{\tilde{N}}^{(l-1,n)}_0$. So they can clearly be coupled to ensure that $\overleftarrow{\tilde{N}}^{(l,n)}_m\ge \overleftarrow{\tilde{N}}^{(l-1,n)}_m$ for all $m$: just let them stick together if they meet. So $\overleftarrow{\tilde{N}}^{(l,n)}_\bullet \sge \overleftarrow{\tilde{N}}^{(l-1,n)}_\bullet$. Since these are just time reversals, this is exactly the statement of the lemma.
\end{proof}

Next we show that if more buses are nonempty, then the first passenger is more likely to travel alone.

\begin{lemma}\label{lem:first_passenger_lonely}
 For any $n\ge l>1$
 \begin{equation}
 \prob^\NE_{l,n}\left(B^{(l)}_1\notin \left\{B^{(l)}_2,\dots,B^{(l)}_n\right\}\right) \ge \prob^\NE_{l-1,n}\left(B^{(l-1)}_1\notin \left\{B^{(l-1)}_2,\dots,B^{(l-1)}_n\right\}\right).
 \end{equation}
\end{lemma}

\begin{proof}
 Let us construct the Markov process $(B^{(l)}_1,\dots,B^{(l)}_n)$ (under the probability $\prob^\NE_{l,n}$, so this is not an i.i.d. sequence) the following way:
  \begin{enumerate}
   \item first  generate the sequence $\tilde{N}^{(l,n)}_\bullet$ (which is the same as deciding when to put a passenger on a new bus),
   \item then decide where to seat each passenger one by one, choosing in each step uniformly from the nonempty or the empty buses, depending on the sequence $\tilde{N}^{(l,n)}_\bullet$ chosen.
  \end{enumerate}
 This construction can be seen to yield the correct distribution for the passenger configuration by symmetry. 
 Then, the conditional probability of the first passenger remaining alone all the time, conditioned on $\tilde{N}^{(l,n)}_\bullet$ is
 \begin{equation}\label{eq:first_pasenger_lonely_conditional_prob}
  \prob^\NE_{l,n}\left(B^{(l)}_1\notin \left\{B^{(l)}_2,\dots,B^{(l)}_n\right\} \,\middle|\, \tilde{N}^{(l,n)}_\bullet \right) =
  \prod_{m\in M} \left(1-\frac{1}{\tilde{N}^{(l,n)}_m}\right)
 \end{equation}
 where
 \begin{equation}
  M=M\left(\tilde{N}^{(l,n)}_\bullet\right)=\left\{m\in\{2,3,\dots,n\} \,\middle|\, \tilde{N}^{(l,n)}_m=\tilde{N}^{(l,n)}_{m-1} \right\}
 \end{equation}
 is the set of time moments when no new bus is taken, meaning that passenger $m$ takes a bus which is already nonempty. Indeed, on such an occasion the first passenger gets a companion with probability $\nicefrac{1}{\tilde{N}^{(l,n)}_m}$, (conditionally) independently of what happened before (conditioned on $\tilde{N}^{(l,n)}_\bullet$, of course).
 
 On the RHS of (\ref{eq:first_pasenger_lonely_conditional_prob}),
 \begin{equation}
  R(\underline{i}):=\prod_{m\in M(\underline{i})} \left(1-\frac{1}{i_m}\right)
 \end{equation}
 is monotone increasing in $\underline{i}\in \IN^{n+1}$ (note that $\underline{i}$ must be a pure birth sequence with $i_0=0$ and $i_1=1$).
 Indeed, if $\underline{j}\ge\underline{i}$ then $R(\underline{j})$ contains at most as many factors as $R(\underline{i})$ (actually exactly $1$ less in the interesting case $j_n=l,i_n=l-1$) (each factor is less than $1$), and each factor is at least as big as the corresponding factor in $R(\underline{i})$.
 
 The law of total probability gives that
 \begin{equation}\label{eq:first_pasenger_lonely_total_prob}
  \prob^\NE_{l,n}\left(B^{(l)}_1\notin \left\{B^{(l)}_2,\dots,B^{(l)}_n\right\} \right) = \expect R(\tilde{N}^{(l,n)}_\bullet)
 \end{equation}
  where, again, $R$ is monotone increasing in the sense that $R(\underline{j})\ge R(\underline{i})$ whenever both make sense and $\underline{j}\ge \underline{i}$. Note that in this expression the function $R$ does not directly depend on $l$, and the expectation $\expect$ does not need to be indexed by $n$ and $l$ because the notation $\tilde{N}^{(l,n)}_\bullet$ already carries the information about the distribution of the process by Convention~\ref{conv:Ntilde}.
  
  Now Lemma~\ref{lem:nonempty_bus_dominance_2} gives us a coupling where $\tilde{N}^{(l,n)}_\bullet\ge \tilde{N}^{(l-1,n)}_\bullet$ almost surely, so $R\left(\tilde{N}^{(l,n)}_\bullet\right)\ge R\left(\tilde{N}^{(l-1,n)}_\bullet\right)$ almost surely as well. The statement of the lemma follows form (\ref{eq:first_pasenger_lonely_total_prob}).
\end{proof}

\begin{remark}
 This proof could be modified to give stochastic dominance between the number of fellow passengers of passenger $1$ in the two systems, by coupling. However, as here, one would have to be careful about comparing only time moments when $N_m$ does not grow -- a kind of time shift which is different for the two processes, to get them synchronized.
\end{remark}
\begin{remark}
 After hearing the result, Péter Csikvári gave a purely combinatorial proof, see Section~\ref{sec:Stirling}.
\end{remark}

Let's get back to the evolution of the number of nonempty buses as passengers arrive. We compare the cases of $l$ and $l-1$ nonempty buses by time $n$ -- i.e. the processes $\tilde{N}^{(l,n)}_\bullet$ and $\tilde{N}^{(l-1,n)}_\bullet$. We aim at a precise understanding of their relation. Since $\tilde{N}^{(l,n)}_0=\tilde{N}^{(l-1,n)}_0=0$, $\tilde{N}^{(l,n)}_n=l$ and $\tilde{N}^{(l-1,n)}_n=l-1$, it is obvious that the difference $\tilde{N}^{(l,n)}_m-\tilde{N}^{(l-1,n)}_m$ grows up from $0$ to $1$ as $m$ goes from $0$ to $n$. We also know from Lemma~\ref{lem:nonempty_bus_dominance_2} that, with a suitable coupling, the difference can be chosen to be always nonnegative. We now show an even better coupling which ensures that the increase from $0$ to $1$ happens in the simplest possible way: monotonically -- meaning that the difference increases exactly once, and stays constant at all other times.

\begin{lemma}\label{lem:nonempty_bus_dominance_3}
 For any $n\ge l>1$ the processes $\tilde{N}^{(l,n)}_\bullet$ and $\tilde{N}^{(l-1,n)}_\bullet$ can be coupled such that the difference process $\tilde{N}^{(l,n)}_\bullet-\tilde{N}^{(l-1,n)}_\bullet$ grows monotonically from $0$ to $1$.
\end{lemma}

\begin{proof}
 The proof is a modification of the proof of Lemma~\ref{lem:nonempty_bus_dominance_2}, so we don't reintroduce the notation. There we saw that the time reversed processes
 $\overleftarrow{\tilde{N}}^{(l,n)}_\bullet$ and
 $\overleftarrow{\tilde{N}}^{(l-1,n)}_\bullet$ are time inhomogeneous pure death processes with the same transition probabilities, started from $\overleftarrow{\tilde{N}}^{(l,n)}_0=l>l-1=\overleftarrow{\tilde{N}}^{(l-1,n)}_0$. 
 
 Having the same transition probabilities means that if we denote any of the forward processes as $X_m$ and the (time dependent) reverse transition probabilities as
 $P^{\rev,m}_{i,j}:=\prob(X_{m-1}=j\,|\, X_m=i)$, then this $P^{\rev,m}_{i,j}$ does not depend on any $n$ or $l$ or $k$ (as long as $n\ge m$ and $k,l\ge i$). Instead, obviously $P^{\rev,m}_{i,j}=0$ unless $j\in\{i,i-1\}$, and
 \begin{equation}\label{eq:reverse_jump_prob}
  P^{\rev,m}_{i,i-1}=\prob( X_m > X_{m-1} \,|\, X_m=i).
 \end{equation}
 But $X_m$ is the number of nonempty buses, so $X_m > X_{m-1}$ if and only if the $m$th passenger chose an empty bus when she arrived. Equivalently, $X_{m-1} < X_m $ means that if we remove the $m$th passenger, the number of nonempty buses decreases. Either way we see that $X_m > X_{m-1}$ if and only if the $m$th passenger \emph{travels alone} after arrival (at least until the arrival of the next passenger). But the conditional probability measure describing the arrival process up to time $m$ under the condition that there are exactly $i$ nonempty buses at time $m$ (which is $\{X_m=i\}$) already has the name $\prob^\NE_{i,m}$ by (\ref{eq:prob-NE-A}), so (\ref{eq:reverse_jump_prob}) means
 \begin{equation}
  P^{\rev,m}_{i,i-1} = \prob^\NE_{i,m} (\{\text{passenger number $m$ travels alone}\}).
 \end{equation}
 By symmetry, the probability of travelling alone is the same for all passengers, so it could just as well be the first passenger. Formally, using the notation of Section~\ref{sec:noempty}:
 \begin{equation}
  P^{\rev,m}_{i,i-1} = \prob^\NE_{i,m} (\{\text{passenger number $1$ travels alone}\})=\prob^\NE_{i,m}\left(B^{(i)}_1\notin \left\{B^{(i)}_2,\dots,B^{(i)}_m\right\}\right).
 \end{equation}
 Lemma~\ref{lem:first_passenger_lonely} allows us to compare these probabilities for different values of $i$, yielding
 \begin{equation}
  P^{\rev,m}_{i,i-1} \ge P^{\rev,m}_{i-1,i-2}.
 \end{equation}
 
 Now we are ready to construct the coupling between our two pure death processes 
 $\overleftarrow{\tilde{N}}^{(l,n)}_\bullet$ and
 $\overleftarrow{\tilde{N}}^{(l-1,n)}_\bullet$:
 \begin{enumerate}
  \item Let them stick together when they meet: this is possible because they have the same transition probabilities;
  \item If they have not yet met, and the smaller one jumps left, let the bigger one also jump left: this is possible because the probability of jumping left is always bigger (or equal) at bigger positions. 
 \end{enumerate}
 This coupling ensures that the difference $\overleftarrow{\tilde{N}}^{(l,n)}_\bullet - \overleftarrow{\tilde{N}}^{(l-1,n)}_\bullet$, which is initially $1$, will never increase, it will decrease to zero on exactly one occasion and then stay zero until the end. This description of the reverse processes is exactly what the lemma states.
\end{proof}

We are ready to prove stochastic monotonicity of the number of lonely passengers as a function of the number of nonempty buses. To avoid misunderstandings, let
$\tilde{L}^{(l,n)}_\bullet$ denote the restriction of $L^{(l)}_\bullet$ to $\Omega^{(l,n,\NE)}$ viewed as a stochastic process on the probability space
$(\Omega^{(l,n,\NE)},\prob^\NE_{l,n})$ (see Section~\ref{sec:noempty} for the notation), so we are looking at the evolution of the number of lonely passengers, as passengers arrive, conditioned on the event that there are exactly $l$ nonempty buses in the end (which is after $n$ passengers).

\begin{proposition}\label{prop:L_l_n_dominance}
 For any $1<l\le n$, $\tilde{L}^{(l,n)}_\bullet \sge \tilde{L}^{(l-1,n)}_\bullet$.
\end{proposition}

\begin{proof}
 Like in the proof of Lemma~\ref{lem:first_passenger_lonely}, we
 construct the Markov process $\left(B^{(l)}_1,\dots,B^{(l)}_n\right)$ under the probability $\prob^\NE_{l,n}$ by first generating the sequence $\tilde{N}^{(l,n)}_\bullet$ (which is the same as deciding when to put a passenger on a new bus), and then deciding where to seat each passenger one by one, choosing in each step uniformly from the nonempty or the empty buses, depending on the sequence $\tilde{N}^{(l,n)}_\bullet$ chosen. Then, at each step (from time $m$ to $m+1$)
 \begin{enumerate}
  \item if $\tilde{N}^{(l,n)}$ grows (by $1$), then $\tilde{L}^{(l,n)}$ also grows (by $1$);
  \item if $\tilde{N}^{(l,n)}$ does not grow, then $\tilde{L}^{(l,n)}$
   \begin{enumerate}
    \item decreases by $1$ with probability $\frac{\tilde{L}^{(l,n)}_m}{\tilde{N}^{(l,n)}_m}$
    \item and stays constant with the remaining probability.
   \end{enumerate}
 \end{enumerate}
 Now we construct a coupling between $\tilde{L}^{(l,n)}_\bullet$ and $\tilde{L}^{(l-1,n)}_\bullet$ by building the two systems simultaneously, making sure that $\tilde{L}^{(l,n)}_m \ge \tilde{L}^{(l-1,n)}_m$ for all $m$. This holds for $m=0$ since both sides are $0$. 
 Let us couple $\tilde{N}^{(l,n)}_\bullet$ and $\tilde{N}^{(l-1,n)}_\bullet$ as in Lemma~\ref{lem:nonempty_bus_dominance_3} and then start the construction of the passenger configurations. In this case
 \begin{enumerate}
  \item whenever $\tilde{N}^{(l-1,n)}$ grows, $\tilde{N}^{(l,n)}$ also grows, which means that whenever $\tilde{L}^{(l-1,n)}$ grows, so does $\tilde{L}^{(l,n)}$. Thus $\tilde{L}^{(l-1,n)}$ has no chance to overtake when growing: it could only ever become bigger than $\tilde{L}^{(l,n)}$ when that decreases.
  \item However, if $\tilde{L}^{(l,n)}$ is bigger (by at least $1$) and it decreases (by $1$), that is no problem: the worst thing that can happen is that they become equal ($\tilde{L}^{(l-1,n)}$ can not increase on such occasions).
  \item On the other hand, if $\tilde{L}^{(l,n)}_m=\tilde{L}^{(l-1,n)}_m=:L$ and there is a chance for $\tilde{L}^{(l,n)}$ to decrease (because $\tilde{N}^{(l,n)}$ does not grow), then the probability for $\tilde{L}^{(l-1,n)}$ to decrease is at least as big:
  \begin{align}
      \prob\left(\tilde{L}^{(l-1,n)}_{m+1}<\tilde{L}^{(l-1,n)}_m \,\middle|\, \{\text{all these}\}\right) & =\frac{L}{\tilde{N}^{(l-1,n)}_m} \\
      & \ge \frac{L}{\tilde{N}^{(l,n)}_m} = \prob\left(\tilde{L}^{(l,n)}_{m+1}<\tilde{L}^{(l,n)}_m \,\middle|\, \{\text{all these}\}\right)
  \end{align}
  because $\tilde{N}^{(l,n)}_m\ge \tilde{N}^{(l-1,n)}_m$ (by Lemma~\ref{lem:nonempty_bus_dominance_3}).
 \end{enumerate}
 So we can couple the constructions so that whenever $\tilde{L}^{(l,n)}_m=\tilde{L}^{(l-1,n)}_m$ and $\tilde{L}^{(l,n)}$ decreases, then so does $\tilde{L}^{(l-1,n)}$, meaning that $\tilde{L}^{(l-1,n)}$ can never become bigger.
\end{proof}

Having solved the main difficulty, we are ready to complete the proof of the main theorem.

\begin{proof}[Proof of Theorem~\ref{thm:L_k_n_dominance}]
 We will first prove the non-strict stochastic dominance $L^{(k+1)}_n \sge L^{(k)}_n$, meaning $\prob_{k+1,n}(L^{(k+1)}_n\ge u) \ge \prob_{k,n}(L^{(k)}_n\ge u)$ for every $u\in\IN$. Then Lemma~\ref{lem:prob_not_equal} will do the rest of the job.
 
 By the law of total probability,
 \begin{equation}
  \prob_{k,n}(L^{(k)}_n\ge u) = \sum_{l=1}^\infty \prob_{k,n}(N^{(k)}_n=l)\prob_{k,n}(L^{(k)}_n\ge u \,|\, N^{(k)}_n=l).
 \end{equation}
 
 If we apply Lemma~\ref{lem:conditioning_on_N_equals_l} to the function $f$ on $\{1,2,\dots,k\}^n$ which maps the passenger configuration to the final number of lonely passengers $L^{(k)}_n$, then we get that
 \begin{equation}
  \prob_{k,n}(L^{(k)}_n\ge u \,|\, N^{(k)}_n=l) = \prob^\NE_{l,n} (L^{(k)}_n\ge u) = \prob(\tilde{L}^{(l,n)}_n\ge u),
 \end{equation}
 since the restriction of $L^{(k)}_n$ to $\Omega^{(k,n,NE)}$ is $\tilde{L}^{(l,n)}_n$, and the indices $k$ and $l$ can be omitted from the probability $\prob$, because the notation $\tilde{L}^{(l,n)}_n$ specifies the distribution. Putting these together,
 \begin{equation}\label{eq:P_L_ge_u}
  \prob_{k,n}(L^{(k)}_n\ge u) = \sum_{l=1}^\infty \prob_{k,n}(N^{(k)}_n=l)\prob(\tilde{L}^{(l,n)}_n\ge u).
 \end{equation}

 With $u$ and $n$ fixed let $g: l \mapsto \prob(\tilde{L}^{(l,n)}_n\ge u)$. With this notation, Proposition~\ref{prop:L_l_n_dominance} says that $g:\IN\to\IR$ is non-decreasing, and (\ref{eq:P_L_ge_u}) says that
 \begin{equation}\label{eq:P_L_ge_u_as_expectation}
  \prob_{k,n}(L^{(k)}_n\ge u) = \sum_{l=1}^\infty \prob_{k,n}(N^{(k)}_n=l) g(l) = \expect g(N^{(k)}_n) \quad \text{ for every $k$}.
 \end{equation}
 Here the expectation $\expect$ need not be indexed by $k$ and $n$, because the notation $N^{(k)}_n$ specifies the distribution.
 
 Lemma~\ref{lem:nonempty_bus_dominance_1} ensures the existence of a coupling such that $N^{(k+1)}_n\ge N^{(k)}_n$ almost surely, so $g(N^{(k+1)}_n)\ge g(N^{(k)}_n)$ almost surely as well. Then (\ref{eq:P_L_ge_u_as_expectation}) gives 
 \begin{equation}
   \prob_{k+1,n}(L^{(k+1)}_n\ge u) \ge \prob_{k,n}(L^{(k)}_n\ge u),
 \end{equation}
 which is the non-strict version of the main statement of the theorem, as said at the start.
 
 In the special case $u=1$ this implies $p_{n,k+1}\ge p_{n,k}$. Then the strict inequalities in the statement of the theorem follow from Lemma~\ref{lem:prob_not_equal}.
 \end{proof}

\renewcommand\thesection{\Alph{section}}
\setcounter{section}{0}
\section{Appendix}

\subsection{Motivation}
\label{sec:motivation}

The problem studied in this paper was designed as a toy example to feature a key difficulty in the following, still open problem of László Márton Tóth.

Consider a discrete time random walk of finite length $n$ on a tree. We say that a \emph{regeneration occurs} if there is an edge which is traversed exactly once. Let $T_3$ denote the $3$-regular tree, and let $T$ denote any tree where every vertex has degree \emph{at least} $3$. Let $X_0,X_1,\dots,X_n$ be a simple symmetric random walk on $T_3$, and let $Y_0,Y_1,\dots,Y_n$ be a simple symmetric random walk on $T$. Prove or disprove the following

\begin{conjecture}[László Márton Tóth, 2023]
For any fixed $n$ as above,
\begin{equation}
\prob(\text{A regeneration occurs for $Y$})\ge \prob(\text{A regeneration occurs for $X$})
\end{equation}
or, more generally,
\begin{equation}
 (\text{number of regenerations for $Y$}) \sge (\text{number of regenerations for $X$}).
\end{equation}
\end{conjecture}

This conjecture itself is motivated by the theory of Ramanujan graphs.
A positive answer would imply that the Bernoulli graphings of
unimodular Galton-Watson trees are Ramanujan in the sense of \cite{BSzV}. (For the notion of Bernoulli graphings see e.g. \cite{Lovasz} 18.3.4. For the notion of the unimodular Galton-Watson trees see e.g. \cite{AldousLyons} Example 1.1)

\subsection{Relation to combinatorics: Stirling numbers of the second kind}
\label{sec:Stirling}
\newcommand{\stirlingii}{\genfrac{\{}{\}}{0pt}{}}

Let $\stirlingii{n}{k}$ be the Stirling number of the second kind, denoting the number of partitions of $\{1,2,\dots,n\}$ into $k$ nonempty sets. So the probability that passenger $1$ travels alone given that there are $n$ passengers in total, under the condition that there are exactly $k$ nonempty buses is $\nicefrac{\stirlingii{n-1}{k-1}}{\stirlingii{n}{k}}$. So Lemma~\ref{lem:first_passenger_lonely} says that
\begin{theorem}
\begin{equation}\label{eq:Stirling_thm}
 \frac{\stirlingii{n-1}{k-1}}{\stirlingii{n}{k}} \le
 \frac{\stirlingii{n-1}{k}}{\stirlingii{n}{k+1}}.
\end{equation}
\end{theorem}
This was first pointed out by Ed Crane, and does not seem to be a known property of Stirling numbers of the second kind. However, Péter Csikvári has given a purely combinatorial proof, which I present with his permission.

\begin{proof}
The statement follows by calculation from the following facts:

\begin{enumerate}
 \item \label{fact:one_alone_or_not} $\stirlingii{n}{r}= \stirlingii{n-1}{r-1} + r \stirlingii{n-1}{r}$. This is trivial combinatorics: the number $1$ is either alone or not.
 \item \label{fact:real_rooted} $P_n(x) := \sum_{k=1}^n \stirlingii{n}{k} x^k$ is a real rooted polynomial. It is called a Touchard polynomial, exponential polynomial or Bell polynomial, see e.g. \cite{Boyadzhiev} equation (3.4) and the paragraphs after (2.15). The property was proven by induction in \cite{Harper}. It can also be found as Exercise 62.1 in \cite{Polya_Szego_II} Part five, Chapter 1.
 \item \label{fact:Newton_inequality} If $\sum_{k=1}^m a_k x^k$ is a real rooted polynomial with nonnegative coefficients, then
 \begin{equation} \label{eq:Newton_inequality}
  \frac{a_{k-1}}{\binom{m}{k-1}} \frac{a_{k+1}}{\binom{m}{k+1}} \le
  \left(\frac{a_k}{\binom{m}{k}}\right)^2.
 \end{equation}
This is Newton's inequality, see e.g. \cite{HLP} Theorem 144 (page 104) or \cite{Vondrak}, Theorem 10.2.
\end{enumerate}

Calculation: Applying fact~\ref{fact:one_alone_or_not} to the denominators in (\ref{eq:Stirling_thm}), after simplification we get
\begin{equation}
 (k+1)\stirlingii{n-1}{k-1}\stirlingii{n-1}{k+1}\le k\stirlingii{n-1}{k}^2.
\end{equation}
This follows from
\begin{equation}
 (k+1)(n-k)\stirlingii{n-1}{k-1}\stirlingii{n-1}{k+1}\le k(n-k-1)\stirlingii{n-1}{k}^2,
\end{equation}
which is exactly (\ref{eq:Newton_inequality}) with $a_k:=\stirlingii{n-1}{k}$ and $m=n-1$, so facts~\ref{fact:real_rooted} and \ref{fact:Newton_inequality} ensure that it is true.
\end{proof}



\begin{thebibliography}{1}

\bibitem{AldousLyons}
David Aldous and Russell Lyons, \emph{Processes on unimodular random networks},
  Electron. J. Probab. \textbf{12} (2007), no. 54, 1454--1508. \MR{2354165}

\bibitem{BSzV}
\'Agnes Backhausz, Bal\'azs Szegedy, and B\'alint Vir\'ag, \emph{Ramanujan
  graphings and correlation decay in local algorithms}, Random Structures
  Algorithms \textbf{47} (2015), no.~3, 424--435. \MR{3385741}

\bibitem{Boyadzhiev}
Khristo~N. Boyadzhiev, \emph{Exponential polynomials, {S}tirling numbers, and
  evaluation of some gamma integrals}, Abstr. Appl. Anal. (2009), Art. ID
  168672, 18. \MR{2545183}

\bibitem{HLP}
G.~H. Hardy, J.~E. Littlewood, and G.~P\'olya, \emph{Inequalities}, Cambridge,
  at the University Press,, 1952, 2d ed. \MR{46395}

\bibitem{Harper}
L.~H. Harper, \emph{Stirling behavior is asymptotically normal}, Ann. Math.
  Statist. \textbf{38} (1967), 410--414. \MR{211432}

\bibitem{Lovasz}
L\'aszl\'o Lov\'asz, \emph{Large networks and graph limits}, American
  Mathematical Society Colloquium Publications, vol.~60, American Mathematical
  Society, Providence, RI, 2012. \MR{3012035}

\bibitem{Polya_Szego_II}
G.~P\'olya and G.~Szeg\H o, \emph{Problems and theorems in analysis. {V}ol.
  {II}}, german ed., Springer Study Edition, Springer-Verlag, New
  York-Heidelberg, 1976, Theory of functions, zeros, polynomials, determinants,
  number theory, geometry. \MR{465631}

\bibitem{Swart} J.M. Swart, \emph{Advanced Topics in Markov chains}, Lecture notes (2018)\\
\verb+https://staff.utia.cas.cz/swart/lecture_notes/chain18_03_22b.pdf+

\bibitem{Vondrak} Jan Vondrák, \emph{lecture notes for Non-constructive methods in combinatorics, Lecture 14: Real-rooted Polynomials} (2016)\\
\verb+https://theory.stanford.edu/~jvondrak/MATH233-2016/Math233-lec14.pdf+

\end{thebibliography}


\providecommand{\bysame}{\leavevmode\hbox to3em{\hrulefill}\thinspace}
\providecommand{\MR}{\relax\ifhmode\unskip\space\fi MR }
\providecommand{\MRhref}[2]{%
  \href{http://www.ams.org/mathscinet-getitem?mr=#1}{#2}
}
\providecommand{\href}[2]{#2}

\begin{acks}
I'm grateful to many people who discussed the problem enthusiastically for months, giving many ideas. Special thanks to Márton Balázs and Ed Crane for their endurance and for their advices on the manuscript. Balázs Ráth is responsible for a few key ideas. I thank Péter Csikvári for the proof in Section~\ref{sec:Stirling}. This research was supported by NKFI grant K-142169. I dedicate this work to Edina Verebélyi for her patience and encouragement.
\end{acks}


\end{document}